\documentclass[11pt,reqno]{amsart}
\usepackage{amscd,amsmath,amsopn,amssymb,amsthm,multicol}
\usepackage{hyperref,color}

\DeclareMathOperator{\ad}{ad}
\DeclareMathOperator{\Id}{Id}

\DeclareMathOperator{\diag}{diag}

\DeclareMathOperator{\Ad}{Ad}

\DeclareMathOperator{\Lin}{Lin}

\renewenvironment{proof}[1][Proof]{\textbf{#1.} }
{\ \rule{0.5em}{0.5em}}
\newtheorem{theorem}{Theorem}
\newtheorem{prop}{Proposition}
\newtheorem{lemma}{Lemma}
\newtheorem{corollary}{Corollary}

\theoremstyle{definition}

\newtheorem{remark}{Remark}

\begin{document}

\title
[Geodesic orbit Riemannian spaces with two isotropy summands \dots] {Geodesic orbit Riemannian spaces \\ with two isotropy summands. I}

%\author{Zhiqi~Chen, Yu.G.~Nikonorov}

\author{Zhiqi~Chen}
\address{Zhiqi~Chen\newline
School of Mathematical Sciences and LPMC, Nankai University, \newline
Tianjin 300071, China}
\email{chenzhiqi@nankai.edu.cn}

\author{Yuri\u{i}~Nikonorov}
\address{Yu.\,G. Nikonorov \newline
Southern Mathematical Institute of Vladikavkaz Scientific Centre \newline
of the Russian Academy of Sciences, Vladikavkaz, Markus st. 22, \newline
362027, Russia}
\email{nikonorov2006@mail.ru}

\begin{abstract}
The paper is devoted to the study of geodesic orbit Riemannian spaces that could be characterize by the property that
any geodesic is an orbit of a 1-parameter group of isometries.
The main result is the classification of compact simply connected geodesic orbit Riemannian spaces $G/H$ with
two irreducible submodules in the isotropy representation.

\vspace{2mm} \noindent 2010 Mathematical Subject Classification:
53C20, 53C25, 53C35.

\vspace{2mm} \noindent Key words and phrases: homogeneous Riemannian manifolds, geodesic orbit
spaces, normal homogeneous spaces, naturally reductive spaces, weakly symmetric spaces.
\end{abstract}

\maketitle

\section{Introduction and the main results}

A Riemannian manifold $(M,g)$ is called {\it a  manifold with
homogeneous geodesics or a geodesic orbit manifold} (shortly,  {\it GO-manifold}) if any
geodesic $\gamma $ of $M$ is an orbit of a 1-parameter subgroup of
the full isometry group of $(M,g)$. A Riemannian manifold $(M=G/H,g)$, where $H$ is a compact subgroup
of a Lie group $G$ and $g$ is a $G$-invariant Riemannian metric,
is called {\it a space with homogeneous geodesics or a geodesic orbit space}
(shortly,  {\it GO-space}) if any geodesic $\gamma $ of $M$ is an orbit of a
1-parameter subgroup of the group $G$. All manifolds in this paper are supposed to be connected.
Hence, a Riemannian manifold $(M,g)$ is  a geodesic orbit Riemannian manifold,
if it is a geodesic orbit space with respect to its full connected isometry group. This terminology was introduced in
\cite{KV} by O.~Kowalski and L.~Vanhecke, who initiated a systematic study on such spaces.
In the same paper, O.~Kowalski and L.~Vanhecke classified all GO-spaces
of dimension $\leq 6$. One can find many interesting results  about  GO-manifolds
and its subclasses in \cite{AA, AN, AV, BerNik,
BerNik3, BerNikClif, DuKoNi, Gor96, Tam, Ta}, and in the references
therein.
\smallskip

All homogeneous spaces in this paper are assumed to be almost effective unless otherwise stated.
Let $(G/H, g)$ be a homogeneous Riemannian space. It is well known that there is an $\Ad(H)$-invariant decomposition
\begin{equation}\label{reductivedecomposition}
\mathfrak{g}=\mathfrak{h}\oplus \mathfrak{m},
\end{equation}
where $\mathfrak{g}={\rm Lie }(G)$ and $\mathfrak{h}={\rm Lie}(H)$.
The Riemannian metric $g$ is $G$-invariant and is determined
by an $\Ad(H)$-invariant Euclidean metric $g = (\cdot,\cdot)$ on
the space $\mathfrak{m}$ which is identified with the tangent
space $T_oM$ at the initial point $o = eH$. By $[\cdot, \cdot]$ we denote the Lie bracket in $\mathfrak{g}$, and by
$[\cdot, \cdot]_{\mathfrak{m}}$ its $\mathfrak{m}$-component according to (\ref{reductivedecomposition}). The following is (in the above terms) a well-known criteria of GO-spaces, see other details and useful facts in \cite{Nik2016}.

\begin{lemma}[\cite{KV}]\label{GO-criterion}
A homogeneous Riemannian space   $(G/H,g)$ with the reductive
decomposition  {\rm(\ref{reductivedecomposition})} is a GO-space if and
only if  for any $X \in \mathfrak{m}$ there is $Z \in \mathfrak{h}$ such that
$$([X+Z,Y]_{\mathfrak{m}},X) =0  \text{ for all } Y\in \mathfrak{m}.$$
\end{lemma}

There are some important subclasses of geodesic orbit manifolds.
Indeed, GO-spaces  may be considered as  a natural generalization of Riemannian symmetric spaces.
On the other hand, the class of GO-spaces is much larger than the class of symmetric spaces.
Any homogeneous space
$M = G/H$ of a compact Lie group $G$  admits a Riemannian metric $g$ such that $(M,g)$ is a GO-space.
It suffices to take the metric
$g$ induced by a biinvariant Riemannian metric $g_0$ on the Lie group  $G$ such that
$ (G,g_0) \to (M=G/H, g)$ is a Riemannian submersion
with totally geodesic fibres. Such geodesic orbit space $(M = G/H, g)$ is called  a {\it normal homogeneous space}.
\smallskip

It should be noted also that  any  naturally  reductive  Riemannian manifold is geodesic orbit. Recall that a Riemannian manifold $(M,g)$ is
{\it naturally reductive} if it admits a transitive Lie group $G$ of isometries with a biinvariant pseudo-Riemannian (non necessarily Riemannian) metric $g_0$,
which induces the metric $g$ on $M = G/H$ (see  \cite{Bes} and \cite{KN}). Clearly symmetric spaces and normal homogeneous spaces are naturally reductive.
The first example of non-naturally reductive GO-manifolds had been constructed by A.~Kaplan \cite{Kap}. In \cite{KV}, O.~Kowalski and L.~Vanhecke
classified all geodesic orbit spaces of dimension $\leq 6$. In particular, they proved that every GO-manifold of dimension $\leq 5$ is naturally reductive.
\smallskip

Another important class of GO-spaces consist of weakly symmetric spaces introduced by A.~Selberg \cite{S}.
A homogeneous Riemannian manifold  $(M = G/H, g)$ is a {\it  weakly symmetric space}
if any two points $p,q \in M$ can be interchanged by
an isometry $a \in G$. This  property does not depend on the particular $G$-invariant metric $g$.
Note that weakly symmetric Riemannian manifolds are geodesic orbit by a result of J.~Berndt, O.~Kowalski, and L.~Vanhecke \cite{BKV}.
Weakly symmetric spaces have many interesting properties and are closely related  with spherical spaces, commutative
spaces, and Gelfand pairs etc (see papers \cite{AV, Yakimova} and book \cite{W1} by J.A.~Wolf).
The classification of weakly symmetric reductive homogeneous  Riemannian spaces was given by O.S.~Yakimova \cite{Yakimova} based on the paper \cite{AV}
(see also \cite{W1}).
\smallskip

{\it Generalized normal homogeneous Riemannian manifolds}
({\it $\delta$-homogeneous manifold}, in another terminology) constitute another important
subclass of geodesic orbit manifolds.
All metrics from this subclass are of non-negative sectional curvature and have some other interesting properties
(see details in \cite{BerNik, BerNik3, BerNik2012}).
In the paper \cite{BerNik2012}, the classification of  generalized normal
homogeneous metrics on spheres and projective spaces is obtained.
Finally, we notice  that {\it Clifford--Wolf homogeneous
Riemannian manifolds} constitute a partial subclass of generalized normal homogeneous Riemannian manifolds \cite{BerNikClif}.
\smallskip

Include the above examples, every isotropy irreducible Riemannian space is naturally reductive, and hence geodesic orbit, see e.g.~\cite{Bes}.
A natural problem is to {\it classify geodesic orbit Riemannian spaces $(G/H, \rho)$ such that the isotropy representation
$\chi:H \rightarrow O(\mathfrak{m})$, $\chi(a)=\Ad(a)|_{\mathfrak{m}}$ has exactly two irreducible components}. This paper is to solve the above problem when
{\it $G/H$ is compact and simply connected}. The general problem need special tools and will be considered in the next paper.
Indeed, the property to be geodesic orbit is related to classes of locally isomorphic homogeneous spaces due to Lemma \ref{GO-criterion}.
First we have the following proposition.
\smallskip

\begin{prop}\label{case5.1p}
Suppose a compact homogeneous space $G/H$ with connected compact $H$ has two irreducible components in the isotropy representation. Then one of the following possibilities holds:
\begin{enumerate}
\item $\mathfrak{g}=\mathfrak{f}\oplus \mathfrak{f}\oplus \mathfrak{f}$ and $\mathfrak{h}=\diag(\mathfrak{f})\subset \mathfrak{g}$
for a compact simple Lie algebra $\mathfrak{f}$;
\item $\mathfrak{g}=\mathfrak{g}_1\oplus \mathfrak{g}_2$ and $\mathfrak{h}=\mathfrak{h}_1\oplus \mathfrak{h}_2$,
where $\mathfrak{h}_i \subset \mathfrak{g}_i$ and the pair $(\mathfrak{g}_i,\mathfrak{h}_i)$ is isotropy irreducible with simple compact Lie algebras
$\mathfrak{g}_i$ for $i=1,2$;
\item $\mathfrak{g}=\mathfrak{f}\oplus \mathfrak{f}\oplus \mathfrak{g}_1$ for simple compact Lie algebras $\mathfrak{g}_1$ and $\mathfrak{f}$,
$\mathfrak{h}=\diag(\mathfrak{f}) \oplus \mathfrak{h}_1$, where $\mathfrak{h}_1 \subset \mathfrak{g}_1$ and the pair $(\mathfrak{g}_1,\mathfrak{h}_1)$
is isotropy irreducible;
\item $\mathfrak{g}=\mathfrak{l}\oplus \mathfrak{k}$, where $\mathfrak{l}$ is a simple compact Lie algebra,
$\mathfrak{k}$ is either a simple compact Lie algebra or $\mathbb{R}$, and there exist
a Lie algebra $\mathfrak{k}_1$ such that
$\mathfrak{k}\oplus \mathfrak{k}_1$ is a subalgebra in $\mathfrak{l}$ such that the pair $(\mathfrak{l}, \mathfrak{k}\oplus \mathfrak{k}_1)$
is isotropy irreducible, whereas $\mathfrak{h}=\diag(\mathfrak{k})\oplus \mathfrak{k}_1 \subset \mathfrak{l}\oplus \mathfrak{k}$;
\item $\mathfrak{g}=\mathbb{R}^2$, $\mathfrak{h}=0$, $G/H=G=S^1\times S^1$;
\item $\mathfrak{g}=\mathbb{R} \oplus \mathfrak{g}_1$, where $\mathfrak{g}_1$ is a semisimple compact Lie algebra,
$\mathfrak{h} \subset \mathfrak{g}_1$ and the pair $(\mathfrak{g}_1,\mathfrak{h})$ is isotropy irreducible for $i=1,2$;
\item $\mathfrak{g}$ is a simple compact Lie algebra.
\end{enumerate}
\end{prop}

Furthermore we have the following theorem:

\begin{theorem}\label{maimres0}
Assume that $G/H$ is a compact and simply connected homogeneous space with non-simple $G$ and the isotropy representation is the direct sum
of two irreducible representations {\rm(}it corresponds to cases {\rm (1)--(6)}
in Proposition {\rm \ref{case5.1p})}. Then $G/H$, supplied with any $G$-invariant Riemannian metric, is naturally reductive, hence, geodesic orbit.
\end{theorem}

In fact, those spaces $G/H$ in cases (2), (3), (5), and (6) of Proposition \ref{case5.1p} are normal homogeneous.
It is known also that $G/H$ in case (4) are naturally reductive (see details in Proposition \ref{variant4} below).
Case (1) of Proposition \ref{case5.1p} is more complicated.
For this case, we prove that all invariant metrics on the Ledger--Obata spaces $F^3/ \diag(F)$,
where $F$ is any connected simple compact Lie group, are naturally reductive, see  Proposition \ref{case5.1pp} below.
Note that for  a simply connected compact homogeneous space $G/H$ the group $H$ is connected, but the cases (5) and (6) are impossible.
\smallskip

Finally, we have the following theorem for $G$ simple.

\begin{theorem}\label{maimres}
Assume that $G/H$ is a compact and simply connected homogeneous space with $G$ simple and the isotropy representation is the direct sum
of two irreducible representations. If $G/H$, supplied with an $G$-invariant Riemannian metric which isn't normal homogeneous with respect to $G$,
is a geodesic orbit space, then there exists $K\subset G$ such that $H\subset K\subset G$, $G/K$ is symmetric, and $(H,K,G)$ is one of the following cases:
\begin{enumerate}
  \item $G_2\subset Spin(7) \subset Spin(8)$;
  \item $SO(2)\times G_2\subset SO(2)\times SO(7) \subset SO(9)$;
  \item $U(k)\subset SO(2k) \subset SO(2k+1)$ for $k\geq 2$;
  \item $SU(2r+1)\subset U(2r+1) \subset SO(4r+2)$ for $r\geq 2$;
  \item $Spin(7)\subset SO(8) \subset SO(9)$;
  \item $SU(m)\times SU(n)\subset S(U(m)U(n))\subset SU(m+n)$ for $m>n\geq 1$;
  \item $Sp(n)U(1)\subset S(U(2n)U(1))\subset SU(2n+1)$ for $n\geq 2$;
  \item $Sp(n)U(1)\subset Sp(n)\times Sp(1)\subset Sp(n+1)$ for $n\geq 1$;
  \item $Spin(10)\subset Spin(10)SO(2)\subset E_6$.
\end{enumerate}
Moreover, every such space $G/H$, supplied with any $G$-invariant Riemannian metric, is geodesic orbit. The space of $G$-invariant Riemannian metrics on $G/H$ has dimension $2$ for all spaces except case {\rm(1)}, where it has dimension $3$.
\end{theorem}

Note that the spaces of cases (1), (3)--(9) in Theorem \ref{maimres} are weakly symmetric, whereas
the space of case (2) is not, see details in  \cite{Ngu}, \cite{W1} or \cite{Yakimova}. The spaces of cases (3), (5), (6) for $n=1$, and (7)
in the above theorem admits generalized normal homogeneous metrics, see details in \cite{BerNik} and \cite{BerNik2012}. Finally, the spaces of cases (6) and (8) in the above theorem are naturally reductive according to Theorem 3 in \cite{Z2}.
\smallskip

It should be noted also that simply connected compact homogeneous spaces $G/H$ with simple $G$ and two
components in the isotropy representation (that corresponds to case (7) of Proposition \ref{case5.1p})
are classified by W.~Dickinson and M.~Kerr in the paper \cite{DickKerr}. In order to prove Theorem~\ref{maimres},
we use this classification substantially together with special new methods. Another important ingredient is the classification
of geodesic orbit spaces fibered over irreducible symmetric spaces by H.~Tamaru \cite{Ta}.
\smallskip

The paper is organized as follows. In Section 2, we classify compact homogeneous spaces $G/H$ with connected $H$ and two modules in the isotropy representation,
i.~e. we prove Proposition~\ref{case5.1p}. Then in Section 3, we study all invariant metrics
on homogeneous manifolds $G/H$ for non-simple $G$ , i.~e. cases (1)--(6) in Proposition~\ref{case5.1p}.
That is, we prove Theorem~\ref{maimres0}. Section 4 is to study case (7) in Proposition~\ref{case5.1p}.
First we develop a theory for a homogenous space of case (7) to be geodesic orbit.
Then we prove Theorem~\ref{maimres} based on the classification given by W.~Dickinson and M.~Kerr in the paper \cite{DickKerr}
and the study on principal isotropy groups of representations in the paper \cite{Hsiang}.

\section{On two irreducible submodules in the isotropy representation}

In this section we prove Proposition~\ref{case5.1p}, hence, describe the structure of compact homogeneous spaces $G/H$ with
connected $H$ and two modules in the isotropy representation.
\smallskip

{\bf The proof of Proposition~\ref{case5.1p}.}
Recall that the properties of a module $\mathfrak{q}\subset \mathfrak{p}$ to be $\Ad(H)$-invariant and
$\ad(\mathfrak{h})$-invariant are equivalent for a connected group $H$.
\smallskip

Since $G/H$ is compact, we know that the Lie group $G$ is compact and $\mathfrak{g}$ is a direct sum of its center $\mathfrak{c(g)}$
and a semisimple ideal $[\mathfrak{g},\mathfrak{g}]$. Since there are two irreducible submodules in the isotropy representation,
we have $\dim \mathfrak{c(g)} \leq 2$.
If $\dim \mathfrak{c(g)} =2$, then $\mathfrak{h}=0$ and we get case~(5) in Proposition~\ref{case5.1p}, because $H$ is trivial due to connectedness.
\smallskip

If $\dim \mathfrak{c(g)} =1$, then we have two possibilities: $\mathfrak{h}\subset \mathfrak{g_1}:=[\mathfrak{g},\mathfrak{g}]$ or
$\mathfrak{h}\not\subset \mathfrak{g_1}$.
In the first case we get case (6) in Proposition~\ref{case5.1p}, in the second we get case (4) in Proposition~\ref{case5.1p} with
$\mathfrak{k}=\mathbb{R}$. It is clear that $\mathfrak{g_1}$ is either simple or $\mathfrak{g_1}=\mathfrak{f} \oplus \mathfrak{f}$ and
$\mathfrak{h}=\diag(\mathfrak{f})$ for some simple Lie algebra $\mathfrak{f}$ in case (6).
\smallskip

Now, we assume that $\mathfrak{c(g)}$ is trivial. We will use the ideas from Section 3 in \cite{Nik2016n}.
Denote by $B=B(\boldsymbol{\cdot}\,,\boldsymbol{\cdot})$ the Killing form of $\mathfrak{g}$. Since $G$ is compact and semisimple,
$B$ is negatively definite on $\mathfrak{g}$.
Therefore, $\langle\boldsymbol{\cdot}\,,\boldsymbol{\cdot}\rangle:=-B(\boldsymbol{\cdot}\,,\boldsymbol{\cdot})$ is a positive definite inner product
on $\mathfrak{g}$.
Since the Lie algebra $\mathfrak{g}$ is semisimple, we can decompose it into a ($\langle\boldsymbol{\cdot}\,,\boldsymbol{\cdot}\rangle$-orthogonal)
sum of simple ideals
$$
\mathfrak{g}=\mathfrak{g}_1\oplus\mathfrak{g}_2\oplus \cdots \oplus \mathfrak{g}_s.
$$
Let $\varphi_i:\mathfrak{h} \rightarrow \mathfrak{g}_i$ be the $\langle\boldsymbol{\cdot}\,,\boldsymbol{\cdot}\rangle$-orthogonal projection.
It is easy to see that all these projections are Lie algebra homomorphisms. We rearrange indices so that $\varphi_i(\mathfrak{h})\neq \mathfrak{g}_i$
for $i=1,2,\dots,p$ and $\varphi_i(\mathfrak{h})= \mathfrak{g}_i$ for $i=p+1,\dots,s$.
\smallskip

Since the Lie algebra $\mathfrak{h}$ is compact, we can decompose it into a
($\langle\boldsymbol{\cdot}\,,\boldsymbol{\cdot}\rangle$-orthogonal) sum of the center and simple ideals
$$
\mathfrak{h}= \mathbb{R}^l\oplus \mathfrak{h}_1\oplus\mathfrak{h}_2\oplus \cdots \oplus \mathfrak{h}_m.
$$
For $i=1,\dots,m$, we denote by $a^i$ the vector $(a^i_1,a^i_2,\dots, a^i_s)\in \mathbb{R}^s$, where
$a^i_j=1$, if $\varphi_j(\mathfrak{h}_i)$ is isomorphic to $\mathfrak{h}_i$, and $a^i_j=0$, if $\varphi_j(\mathfrak{h}_i)$ is a trivial Lie algebra
(there is no other possibility, because $\varphi_j$ is a Lie algebra homomorphism).
It is easy to see that $\sum_{i=1}^m a^i_j=1$ for $j=p+1,\dots,s$, since $\varphi_j(\mathfrak{h})= \mathfrak{g}_j$ is a simple Lie algebra.
Denote also the number $\dim(\varphi_i(\mathbb{R}^l))$ by $u_i$ for $i=1,\dots,s$, and put $u=\sum_{i=1}^s u_i$, $v_i=\sum_{j=1}^s a^i_j$ for $i=1,\dots,m$.
It is clear that $u \geq l$ and $v_i\geq 1$ for all $i$.
\smallskip

Note that (in the above notation and suggestions) the following inequality holds:
$$
p+u+\sum_{i=1}^m v_i - l-m =p+(u-l)+\sum_{i=1}^m (v_i-1) \leq 2.
$$
Indeed, for $i=1,2,\dots,p$, every $\mathfrak{g}_i$ contains at least one irreducible modules $\mathfrak{p}_j \subset \mathfrak{p}$, since
$\varphi_i(\mathfrak{h})\neq \mathfrak{g}_i$ and $\langle\boldsymbol{\cdot}\,,\boldsymbol{\cdot}\rangle$-orthogonal complement to
$\varphi_i(\mathfrak{h})$ in $\mathfrak{g}_i$ is a subset of $\mathfrak{p}$. This gives at least $p$ irreducible modules.
Further, an $\langle\boldsymbol{\cdot}\,,\boldsymbol{\cdot}\rangle$-orthogonal complement to
$\mathbb{R}^l$ in $\oplus_{i=1}^s \varphi_i(\mathbb{R}^l)$ is also a subset of $\mathfrak{p}$. It is clear that $\ad(\mathfrak{h})$ acts trivially
on this complement, hence we get exactly $u-l$ one-dimensional irreducible submodules in it.
Finally, for any $i=1,\dots,m$, an $\langle\boldsymbol{\cdot}\,,\boldsymbol{\cdot}\rangle$-orthogonal complement to
$\mathfrak{h}_i$ in $\oplus_{j=1}^s \varphi_j(\mathfrak{h}_i)$ is also a subset of $\mathfrak{p}$.
In fact, we deal with compliment to $\diag(\mathfrak{h}_i)$ in $\mathfrak{h}_i\oplus \mathfrak{h}_i\oplus \cdots \oplus \mathfrak{h}_i$
($v_i$ pairwise isomorphic summands).
In this case we have exactly $(v_i-1)$ $\ad(\mathfrak{h})$-irreducible modules. Summing all numbers of irreducible submodules, we get the inequality.
\smallskip

Without loss of generality we may rearrange the indices so that $v_1 \geq v_2 \geq \cdots \geq v_{m-1}\geq v_m (\geq 1)$.
\smallskip

First, we prove that for any $j$ with
$v_j=1$ we get $a_j^i=0$ for all $i> p$.
Indeed, if $a_j^i=1$ for some $i >p $, then $\mathfrak{g}_i=\mathfrak{h}_j$
(since $\mathfrak{g}_i=\varphi_i(\mathfrak{h})\supset \varphi_i(\mathfrak{h}_j)\neq 0$, whereas $\varphi_k(\mathfrak{h}_j)=0$ for $k \neq i$ due to $v_j=1$),
that contradicts to the effectiveness of the pair $(\mathfrak{g},\mathfrak{h})$.
Therefore, $v_j=1$ implies $\mathfrak{h}_j \subset \mathfrak{g}_1\oplus \cdots \oplus \mathfrak{g}_p$.
\smallskip

It is clear that $v_1\leq 3$, moreover, $v_1=3$ implies $l=0$, $m=1$ and we get case {\rm (1)} of Proposition~\ref{case5.1p} since the pair
$(\mathfrak{h}_1\oplus \mathfrak{h}_1\oplus \mathfrak{h}_1,\diag(\mathfrak{h}_1))$ determines the homogeneous space with two irreducible submodules.
\smallskip

If $v_1=2$ and $v_2=2$, then $v_j=1$ for all $j \geq 3$, $p=0$ and $u=l$. It is easy to see that $m=2$, $l=0$,
$\mathfrak{h}_i=\diag(\mathfrak{f}_i) \subset \mathfrak{f}_i \oplus \mathfrak{f}_i$ for a simple compact Lie algebra $\mathfrak{f}_i$, $i=1,2$, and
$\mathfrak{g}= \mathfrak{f}_1 \oplus \mathfrak{f}_1 \oplus  \mathfrak{f}_2 \oplus \mathfrak{f}_2$, hence we get case {\rm (2)} of Proposition~\ref{case5.1p}.
\smallskip

If $v_1=2$ and  $v_j=1$ for all $j \geq 2$, then $p\leq 1$ and (as it is proved above)
we get $a_j^i=0$ for all $i\geq 2$.
Therefore, $a_j^1=1$ and $\mathfrak{h}_j \subset \mathfrak{g}_1$ for all $i\geq 2$.
If $a_1^1=0$ then we get case {\rm (3)}, if $a_1^1=1$ then we get case {\rm (4)} of Proposition~\ref{case5.1p}.
\smallskip

Finally, suppose that  $v_j=1$ for all $j \geq 1$. Clear that $p\leq 2$. For any $j$ we get $a_j^i=0$ for all $i>p$
and $\mathfrak{h}_j \subset \mathfrak{g}_1\oplus \cdots \oplus \mathfrak{g}_p$ (see above). Hence, $p=0$ is impossible and
$p=1$ implies that $\mathfrak{h}_j \subset \mathfrak{g}_1$ for all $j$, therefore, $G$ is simple, and we get the {\rm (7)} of the proposition.
If $p=2$, then $u=l$, $\mathfrak{g}=\mathfrak{g}_1\oplus \mathfrak{g}_2$ and $\mathfrak{h}=\mathfrak{k}_1\oplus \mathfrak{k}_2$ with
$\mathfrak{k}_i \subset \mathfrak{g}_i$, whereas the pair $(\mathfrak{g}_i,\mathfrak{h}_i)$ is isotropy irreducible, $i=1,2$.
Hence, we get case {\rm (2)} again.

\begin{remark}
The problem of the classification of compact homogeneous spaces $G/H$ with two components in the isotropy representation
and non-connected $H$, is more complicated and is not studied here.
\end{remark}

\section{The cases for non-simple $G$}

This section is to prove Theorem~\ref{maimres0}. We will discuss case by case. Firstly, for cases (2), (3), (5) and (6) of Proposition \ref{case5.1p},
we deal with the product of two isotropy irreducible spaces, hence normal homogeneous, and therefore geodesic orbit.
Recall also that the cases (5) and (6) are impossible simply connected compact homogeneous spaces $G/H$.
\smallskip

Secondly, for case (4), we get naturally reductive, hence geodesic orbit metrics. This result is well-known (see e.~g. Theorem 3 in \cite{Z2}),
but for reader's convenience we add also a short proof.

\begin{prop}\label{variant4} Suppose that
$\mathfrak{g}=\mathfrak{l}\oplus \mathfrak{k}$, where
where $\mathfrak{l}$ is a simple Lie algebra,
$\mathfrak{k}$ is either a simple Lie algebra or $\mathbb{R}$, and there exists
a Lie algebra $\mathfrak{k}_1$ such that
$\mathfrak{k}\oplus \mathfrak{k}_1$ is a subalgebra in~$\mathfrak{l}$ and the pair $(\mathfrak{l}, \mathfrak{k}\oplus \mathfrak{k}_1)$
is isotropy irreducible, whereas $\mathfrak{h}=\diag(\mathfrak{k})\oplus \mathfrak{k}_1 \subset \mathfrak{l}\oplus \mathfrak{k}$.
Then every $G$-invariant Riemannian metric on the corresponding homogeneous space $G/H$ with connected $H$ is naturally reductive either with respect to
$G=L\times K$ or with respect to $L$.
\end{prop}

\begin{proof}
Let $\langle \cdot,\cdot \rangle$ be the minus Killing form of $\mathfrak{l}$ and suppose that $\mathfrak{k}$
(in the sum $\mathfrak{g}=\mathfrak{l}\oplus \mathfrak{k}$) is supplied with the restriction of
$\langle \cdot,\cdot \rangle$ (we use the fact that $\mathfrak{k}\oplus \mathfrak{k}_1$ is a subalgebra in~$\mathfrak{l}$).
We have $\ad(\mathfrak{h})$-invariant complement $\mathfrak{m}=\mathfrak{m}_1\oplus \mathfrak{m}_2$, where
$\mathfrak{m}_1$ is the $\langle \cdot,\cdot \rangle$-orthogonal complement to $\mathfrak{k}\oplus \mathfrak{k}_1$  in $\mathfrak{l}$ and
$\mathfrak{m}_2=\{(X,-X)\in \mathfrak{k}\oplus \mathfrak{k} \subset \mathfrak{l}\oplus \mathfrak{k}\}$.
Any  $\ad(\mathfrak{h})$-invariant inner product $(\cdot,\cdot)$ has the form
$(\cdot,\cdot)=a\cdot \langle \cdot,\cdot \rangle_{\mathfrak{m}_1}+b\cdot \langle \cdot,\cdot \rangle_{\mathfrak{m}_2}$ for some $a,b >0$.
Now, let us consider an $\ad(\mathfrak{g})$-invariant inner product $\langle \cdot,\cdot \rangle_{\ast}=
\alpha \cdot \langle \cdot,\cdot \rangle +\beta\cdot \langle \cdot,\cdot \rangle$ on $\mathfrak{g}=\mathfrak{l}\oplus \mathfrak{k}$,
$\alpha,\beta \in \mathbb{R}\setminus \{0\}$.
The $\langle \cdot,\cdot \rangle_{\ast}$-orthogonal complement to $\mathfrak{h}$ is $\mathfrak{m}_1 \oplus \widetilde{\mathfrak{m}}_2$, where
$\widetilde{\mathfrak{m}}_2=\{(\beta X,-\alpha X)\in \mathfrak{k}\oplus \mathfrak{k} \subset \mathfrak{l}\oplus \mathfrak{k}\}$.
Clear that $\langle \cdot,\cdot \rangle_{\ast}$ generates $(\cdot,\cdot)$ if and only if $\alpha =a$ (we compare the metrics on $\mathfrak{m}_1$) and
$4\alpha\beta(\alpha+\beta)=(\alpha+\beta)^2(a+b)$ (we compare the metrics on vectors $\bigl((\beta+\alpha) X, -(\beta+\alpha) X\bigr) \in \mathfrak{m}_2$
and on vectors $(2\beta X, -2\alpha X) \in \widetilde{\mathfrak{m}}_2$ for $X\in \mathfrak{k}$, since
$(2\beta X, -2\alpha X)-\bigl((\beta+\alpha) X, -(\beta+\alpha) X\bigr) \in \mathfrak{h}$).
This means that $\langle \cdot,\cdot \rangle_{\ast}$ with $\alpha=a >0$ and $\beta=\frac{a+b}{3a-b}$ does generate the metric $(\cdot,\cdot)$ if $3a \neq \beta$.
In particular $(\cdot,\cdot)$ is naturally reductive with respect to the Lie group $G=L\times K$ with the Lie algebra
$\mathfrak{g}=\mathfrak{l}\oplus \mathfrak{k}$.
\smallskip

If $3a=b$ then $(\cdot,\cdot)$ is generate by the restriction of $a\cdot \langle \cdot,\cdot \rangle$ on $\mathfrak{m}_1\oplus \mathfrak{k} \subset \mathfrak{l}$,
that is also $\ad(\mathfrak{h})$-invariant complement to $\mathfrak{h}$ in $\mathfrak{g}$.
Indeed $(2X,0)-(X,-X)\in \mathfrak{h}$ for $X\in \mathfrak{k}$ and $4a\langle X, X \rangle=(a+b)\langle X, X \rangle$.
In this case $(\cdot,\cdot)$ is naturally reductive with respect to the Lie group $L$.
\end{proof}
\medskip

Finally, for case (1) of Proposition \ref{case5.1p}, we claim that all invariant metrics on $F^3/\diag(F)$ (the Ledger--Obata space)
are naturally reductive, hence, geodesic orbit.
\smallskip

The spaces $F^m/\diag(F)$ are called Ledger--Obata spaces, where $F$ is a connected compact simple Lie group,
$F^m=F\times F\times\cdots\times F$ ($m$ factors and $m\geq 2$), and $\diag(F)=\{(X,X,\dots,X)|X\in F\}$. Ledger--Obata spaces were first introduced
in \cite{LO1968} as a natural generalization of symmetric spaces, since $F^2/\diag(F)$ is an irreducible symmetric space. For more details see \cite{CNN2017}.
Here we deal with the spaces $G/H=F^3/\diag(F)$. Note also that these spaces give examples of generalised Wallach spaces, see \cite{Nik2016n}.
It is not difficult to get that the spaces of all invariant metrics on $F^3/\diag(F)$ and naturally reductive metrics
both have dimension $3$. This hints to the assertion of the following proposition.

\begin{prop}\label{case5.1pp}
For any compact connected simple Lie group $F$, every invariant Riemannian metric on the Ledger--Obata space $G/H=F^3/\diag(F)$ is naturally reductive,
hence geodesic orbit.
\end{prop}

\begin{proof} Note that any irreducible $\Ad(\diag(F))$-invariant submodule in $\mathfrak{g}=\mathfrak{f}\oplus \mathfrak{f}\oplus \mathfrak{f}$ has the form
$\left\{(\alpha_1 X, \alpha_2 X, \alpha_{3} X)\subset \mathfrak{g}\,|\, X \in \mathfrak{f}\right\}$ for some fixed $\alpha_i$,
see details e.~g. in~\cite{CNN2017}.
In particular, $\mathfrak{h}=\diag(\mathfrak{f})=\left\{(X,X,X)\subset \mathfrak{g}\,|\, X \in \mathfrak{f}\right\}$.
We can choose $\Ad(\diag(F))$-invariant complement to $\mathfrak{h}$ in $\mathfrak{g}$ as follows:
$\mathfrak{p}_1=\mathfrak{f}\oplus \mathfrak{f}\oplus 0\subset \mathfrak{g}$.
Any inner product $(\cdot,\cdot)$ on $\mathfrak{p}_1$ is determined by a positive definite form $f(x,y)=Ax^2+2Bxy+Cy^2$, $x,y \in \mathbb{R}$, by the following
equality:
$$
\Bigl((x Z, y Z,0), (x Z, y Z,0) \Bigr)=f(x,y)\cdot \langle Z, Z \rangle, \qquad Z \in \mathfrak{f},
$$
where $\langle \cdot, \cdot \rangle$ is the minus Killing form of $\mathfrak{f}$. Let us fix such $(\cdot,\cdot)$.
If $B=0$, then $(\cdot, \cdot)$ is normal homogeneous (and even bi-invariant) with respect to $F\times F=F\times F\times e \subset F^3$ (since $A>0$ and $C>0$).
Therefore, all metrics with $B=0$ are naturally reductive.
\smallskip

Let us consider also the following two $\Ad(\diag(F))$-invariant complement to $\mathfrak{h}$ in $\mathfrak{g}$:
$\mathfrak{p}_2=\mathfrak{f}\oplus 0\oplus \mathfrak{f}\subset \mathfrak{g}$ and $\mathfrak{p}_3=0\oplus \mathfrak{f}\oplus \mathfrak{f}\subset \mathfrak{g}$.
\smallskip

If $A+B=0$, then $C>B$ (since $AC>B^2=A^2$) and $Ax^2+2Bxy+Cy^2=B(x-y)^2+(C-B)y^2$. Since the vector
$(x,y,0)\in \mathfrak{p}_1$ is projected to $(x-y,0,-y)\in \mathfrak{p}_2$ along  $\mathfrak{h}$, we know that the metric $(\cdot,\cdot)$ for $A+B=0$
is normal homogeneous (and even bi-invariant) with respect to $F\times F=F\times e\times  F \subset F^3$.
\smallskip

If $B+C=0$, then $A>B$ (since $AC>B^2=C^2$) and $Ax^2+2Bxy+Cy^2=C(y-x)^2+(A-C)x^2$. Since the vector
$(x,y,0)\in \mathfrak{p}_1$ is projected to $(0,y-x,-x)\in \mathfrak{p}_3$ along  $\mathfrak{h}$, we know that the metric $(\cdot,\cdot)$ for $B+C=0$
is normal homogeneous (and even bi-invariant) with respect to $F\times F=e\times F\times  F \subset F^3$.
\smallskip

Hence, for $B(A+B)(B+C)=0$ we get naturally reductive metrics $(\cdot,\cdot)$ with respect to some subgroups $F^2$ in $F^3$.
\smallskip

Now, let us assume that $B\neq 0$, $A+B\neq 0$, and $B+C\neq 0$.
We can supply $\mathfrak{g}$ with a 3-parameter family of bi-invariant pseudo-Riemannian metrics of the form
$$
(\cdot, \cdot)_1=\alpha \langle \cdot, \cdot \rangle|_{(\mathfrak{f},0,0)}+\beta \langle \cdot, \cdot \rangle|_{(0,\mathfrak{f},0)}+
\gamma \langle \cdot, \cdot \rangle|_{(0,0,\mathfrak{f})}\,, \qquad \alpha,\beta,\gamma \neq 0.
$$
Let $\mathfrak{p}_1$ is the $(\cdot, \cdot)_1$-orthogonal complement to $\mathfrak{h}$ in $\mathfrak{g}$ (for the case $\alpha+\beta+\gamma\neq 0$), i.~e.
$$
\mathfrak{p}_1=\left\{\Lin \bigl((a X, b X,c X)\bigr) \subset \mathfrak{g}\,|\, X \in \mathfrak{f},\, a\alpha+b\beta+c\gamma=0\right\}.
$$
The restriction of $(\cdot, \cdot)_1$ to $\mathfrak{p}_1$ determines a naturally reductive homogeneous Riemannian metric on $G/H$.
Now, we will prove that there are $\alpha,\beta,\gamma \neq 0$, $\alpha+\beta+\gamma\neq 0$,
such that this metric is the same as the metric generated with $(\cdot,\cdot)$.
If $\pi$ is the projection $\mathfrak{p}_1 \rightarrow \mathfrak{p}$ along $\mathfrak{h}$,
then $\pi \bigl((a X, b X,c X)\bigr)=\bigl((a-c)X, (b-c)X, 0 \bigr)$.
Further,
\begin{eqnarray*}
\Bigr((a X, b X,c X),(a X, b X,c X)\Bigl)_1&=&(a^2\alpha+b^2\beta+c^2\gamma)\cdot \langle X, X \rangle, \\
\Bigl(\big((a-c)X, (b-c)X, 0 \bigr),\bigl( (a-c)X, (b-c)X, 0 \bigr) \Bigr)&=&f((a-c),(b-c))\cdot \langle X, X \rangle,
\end{eqnarray*}
where $X \in \mathfrak{f}$.
These two values are equal (hence, $(\cdot, \cdot)_1$ and $(\cdot, \cdot)$ are isometric)
if and only if $a^2\alpha+b^2\beta+c^2\gamma=f((a-c),(b-c))$. Recall that $c=-(a\alpha+b\beta)/\gamma$, hence we get the equality
\begin{eqnarray*}
&& a^2\alpha\gamma^2+b^2\beta\gamma^2+(a\alpha+b\beta)^2\gamma \\
&=& A((\alpha+\gamma)a+\beta b)^2+2B((\alpha+\gamma)a+\beta b)(\alpha a+(\beta+\gamma)b)+C(\alpha a+(\beta+\gamma)b)^2
\end{eqnarray*}
for every $a,b \in \mathbb{R}$, which is equivalent to the following system (since $a$ and $b$ could be arbitrary):
\begin{eqnarray*}
\alpha \gamma (\alpha+\gamma)&=&A(\alpha+\gamma)^2+2B(\alpha+\gamma) \alpha+ C\alpha^2; \\
\beta \gamma (\beta+\gamma)&=&A\beta^2+2B\beta(\beta+\gamma)+ C(\beta+\gamma)^2; \\
\alpha \beta \gamma&=&A\beta(\alpha+\gamma)+B((\alpha+\gamma)(\beta+\gamma)+\alpha \beta)+ C\alpha(\beta+\gamma).
\end{eqnarray*}
It is easy to see that this system is equivalent to the following one:
\begin{eqnarray*}
A=\frac{\alpha(\beta+\gamma)}{\alpha+\beta+\gamma}, \quad
B=\frac{-\alpha\beta}{\alpha+\beta+\gamma}, \quad
C=\frac{\beta(\alpha+\gamma)}{\alpha+\beta+\gamma}\,,
\end{eqnarray*}
and has the following solution:
$(\alpha, \beta,\gamma)=\left(\frac{D}{B+C}, -\frac{D}{B}, \frac{D}{A+B}\right)$, where $D=AC-B^2>0$. Obviously,
$\alpha \beta \gamma\neq 0$.
Note also that $\alpha+\beta+\gamma=-\frac{D^2}{B(A+B)(B+C)}\neq 0$.
Therefore, all invariant Riemannian metrics with the property $B(A+B)(B+C)\neq 0$
are naturally reductive with respect to $F\times F\times F$.
\end{proof}

\begin{remark}
The classification of geodesic orbit metrics on the Ledger--Obata spaces $F^m/\diag(F)$ for $m \geq 4$ is essentially much more complicated
than that for $m=3$ and is not known by now.
\end{remark}

\section{The case for $G$ simple}
This section is to prove Theorem~\ref{maimres}. The proof is based on the classification of $G/H$ with $G$ simple and two irreducible
submodules in the isotropy representation which is given by W.~Dickinson and M.~Kerr in the paper \cite{DickKerr}. In that paper, they
divide such homogeneous spaces into two classes:
\begin{enumerate}
  \item[(i)] $H$ is maximal in $G$;
  \item[(ii)] there exists a subgroup $K$ of $G$ such that $H\subset K\subset G$.
\end{enumerate}
We will prove Theorem~\ref{maimres} by the following steps.
\begin{enumerate}
   \item Firstly, we give some necessary conditions for $G/H$ to be geodesic orbit.
   \item Furthermore, by the result in (1), we prove that if $G/H$ is geodesic orbit, there exists a subgroup $K$ of $G$ such that $H\subset K\subset G$
   and $G/K$ is symmetric. In fact, such geodesic orbit spaces $G/H$ belong to the list given in \cite{Ta}.
   \item Finally, by comparing the table of \cite{Ta} with the list in \cite{DickKerr}, we get Theorem~\ref{maimres}.
\end{enumerate}

Suppose that $G/H$ is compact, $\langle \cdot,\cdot \rangle$ is an $\Ad(G)$-invariant inner product on $\mathfrak{g}$, $\mathfrak{m}$ is
the $\langle \cdot,\cdot \rangle$-orthogonal complement to $\mathfrak{h}$ in $\mathfrak{g}$.
Let $\mathfrak{m}=\mathfrak{m}_1\oplus \mathfrak{m}_2$, where $\langle \mathfrak{m}_1, \mathfrak{m}_2 \rangle=0$ and both
$\mathfrak{m}_1$ and $\mathfrak{m}_2$ are $\Ad(G)$-invariant. Consider the following two-parameter family of $G$-invariant metrics:
\begin{equation}\label{twopar1}
(\cdot,\cdot)=\lambda \langle \cdot,\cdot \rangle|_{\mathfrak{m}_1}+\mu \langle \cdot,\cdot \rangle|_{\mathfrak{m}_2},\quad \lambda \neq 0,\, \mu \neq 0.
\end{equation}
Such metrics correspond to the metric operators $A=A_{\lambda,\mu}=\lambda \Id|_{\mathfrak{m}_1}+\mu  \Id|_{\mathfrak{m}_2}$, i.~e.
$(\cdot,\cdot)=\langle A \cdot,\cdot \rangle|_{\mathfrak{m}}$.
If $\lambda = \mu$, we always get $G$-normal, hence geodesic orbit metrics.

\begin{remark}
Among compact simply connected spaces $G/H$ with simple $G$ and two irreducible submodules in the isotropy representation,
only the space $Spin(8)/G_2$ (which is a double cover of $SO(8)/G_2$) has two isomorphic irreducible submodules, see \cite{DickKerr,Kerr98}.
Hence any $G$-invariant metric on such spaces $G/H$ except $Spin(8)/G_2$ has the form~(\ref{twopar1}).
\end{remark}

First we study the space $Spin(8)/G_2$, which is diffeomorphic to $S^7 \times S^7$.
Clearly the space of $Spin(8)$-invariant metric on $Spin(8)/G_2$ has dimension $3$, see a detailed exposition in~\cite{Kerr98}.
Moreover, the following is a well-known fact.

\begin{prop}[\cite{Zil96}]\label{prop.zil}
Every $Spin(8)$-invariant metric on $Spin(8)/G_2$ is weakly symmetric, hence, geodesic orbit.
\end{prop}

Note that $Spin(8)/G_2\cong S^7 \times S^7$ admits a one-parameter family of  $Spin(8)$-normal homogeneous metrics and a two-parameter family
of $(SO(8)\times SO(8))$-normal homogeneous metrics. Any other $Spin(8)$-invariant metric on $Spin(8)/G_2$ is not naturally reductive.

\smallskip

In the rest of the section we deal with
compact simply connected spaces $G/H$ with simple $G$ and two irreducible submodules in the isotropy representation except $Spin(8)/G_2$.
Recall that every $G$-invariant metric on every such spaces $G/H$ has the form (\ref{twopar1}).

\begin{lemma}[\cite{AA,AN}]\label{GO-criterion1}
The $G$-invariant Riemannian metric on  $G/H$ with the metric operator $A=A_{\lambda,\mu}$, $\lambda \neq \mu$,  is geodesic orbit if and
only if for any $X \in \mathfrak{m}_1$ and any $Y \in \mathfrak{m}_2$ there is
$Z \in \mathfrak{h}$ such that
$[Z +X+Y, A(X+Y)]\in \mathfrak{h}$, or equivalently $[X,Y]=\frac{\lambda}{\lambda-\mu}[Z,X]+\frac{\mu}{\lambda-\mu}[Z,Y]$.
\end{lemma}

We will denote by $C_{\mathfrak{h}}(U)$ and $N_{\mathfrak{h}}(U)$ the centralizer and the normalizer  of a vector $U\in \mathfrak{g}$ in $\mathfrak{h}$
respectively.
Note that for any $U\in \mathfrak{g}$, we have the following direct sum of Lie algebras:
\begin{equation}\label{normcentr}
N_{\mathfrak{h}}(C_{\mathfrak{h}}(U))=C_{\mathfrak{h}}(U)\oplus \widetilde{C}_{\mathfrak{h}}(U),
\end{equation}
where $\widetilde{C}_{\mathfrak{h}}(U)$ is the $\langle \cdot,\cdot \rangle$-orthogonal complement to $C_{\mathfrak{h}}(U)$ in
$N_{\mathfrak{h}}(C_{\mathfrak{h}}(U))$, in particular, $[\widetilde{C}_{\mathfrak{h}}(U),C_{\mathfrak{h}}(U)]=0$.
\smallskip

We specify Lemma \ref{GO-criterion1}, using the idea of the geodesic graph, see \cite{KV}.
It should be noted that $C_{\mathfrak{h}}(X+Y)=C_{\mathfrak{h}}(X)\cap C_{\mathfrak{h}}(Y)$ for any $X \in \mathfrak{m}_1$ and any $Y \in \mathfrak{m}_2$.

\begin{lemma}\label{GO-criterion2}
The $G$-invariant Riemannian metric on  $G/H$ with the metric operator $A_{\lambda,\mu}$, $\lambda \neq \mu$,  is geodesic orbit if and
only if for any $X \in \mathfrak{m}_1$ and any $Y \in \mathfrak{m}_2$ there is a unique
$Z \in \widetilde{C}_{\mathfrak{h}}(X+Y)$ such that
$[X,Y]=\frac{\lambda}{\lambda-\mu}[Z,X]+\frac{\mu}{\lambda-\mu}[Z,Y]$.
\end{lemma}

\begin{proof} If $Z \in \mathfrak{h}$ such that
$[X,Y]=\frac{\lambda}{\lambda-\mu}[Z,X]+\frac{\mu}{\lambda-\mu}[Z,Y]$, then $Z\in N_{\mathfrak{h}}(C_{\mathfrak{h}}(X+Y))$. Indeed, take any
$U\in C_{\mathfrak{h}}(X+Y)$, then
$$
0=[U,[X,Y]]=\frac{\lambda}{\lambda-\mu}[[U,Z],X]+\frac{\mu}{\lambda-\mu}[[U,Z],Y].
$$
Hence, $[U,Z]\in C_{\mathfrak{h}}(X+Y)$ and $Z\in N_{\mathfrak{h}}(C_{\mathfrak{h}}(X+Y))$. For any $Z_0\in \mathfrak{h}$ such that
$[X,Y]=\frac{\lambda}{\lambda-\mu}[Z_0,X]+\frac{\mu}{\lambda-\mu}[Z_0,Y]$, we have $\frac{\lambda}{\lambda-\mu}[Z-Z_0,X]+\frac{\mu}{\lambda-\mu}[Z-Z_0,Y]=0$ and,
therefore,
$Z-Z_0\in C_{\mathfrak{h}}(X+Y)$.
Hence, the $\widetilde{C}_{\mathfrak{h}}(X+Y)$-component of $Z$ (see (\ref{normcentr})) has the required property.
\end{proof}

\begin{remark} The map $\Xi :\mathfrak{m} \rightarrow \mathfrak{h}$ that sends any $X+Y \in \mathfrak{m}$ to the vector $Z$ from Lemma~\ref{GO-criterion2}
is called the geodesic graph. It is very useful in the study of geodesic orbit metrics, see~\cite{KV}.
\end{remark}

\begin{corollary}\label{cor1}
If the metric with $A_{\lambda,\mu}$,  $\lambda \neq \mu$, is geodesic orbit and $N_{\mathfrak{h}}(C_{\mathfrak{h}}(X+Y))=C_{\mathfrak{h}}(X+Y)$ for some
 $X \in \mathfrak{m}_1$ and $Y \in \mathfrak{m}_2$, then $[X,Y]=0$.
\end{corollary}

\begin{proof} The vector $Z$ in Lemma \ref{GO-criterion2} should be trivial.
\end{proof}

\begin{theorem}\label{two modules}
Let $G/H$ be compact, $\langle \cdot,\cdot \rangle$ is $\Ad(G)$-invariant inner product on $\mathfrak{g}$, $\mathfrak{m}$ is the
$\langle \cdot,\cdot \rangle$-orthogonal complement to $\mathfrak{h}$ in $\mathfrak{g}$. Suppose that
$\mathfrak{m}=\mathfrak{m}_1\oplus \mathfrak{m}_2$, where $\langle \mathfrak{m}_1, \mathfrak{m}_2 \rangle=0$ and both
$\mathfrak{m}_1$ and $\mathfrak{m}_2$ are $\Ad(H)$-invariant.
Then the following conditions are equivalent:
\begin{enumerate}
\item the metric
$(\cdot,\cdot)=\lambda \langle \cdot,\cdot \rangle|_{\mathfrak{m}_1}+\mu \langle \cdot,\cdot \rangle|_{\mathfrak{m}_2}$
is geodesic orbit for some $\lambda\neq \mu$;
\item the metric
$(\cdot,\cdot)=\lambda \langle \cdot,\cdot \rangle|_{\mathfrak{m}_1}+\mu \langle \cdot,\cdot \rangle|_{\mathfrak{m}_2}$
is geodesic orbit for any $\lambda,\mu$;
\item for any $X\in \mathfrak{m}_1$ and $Y\in \mathfrak{m}_2$ there is a unique
$Z_X \in \widetilde{C}_{\mathfrak{h}}(X+Y)\cap C_{\mathfrak{h}}(X)$ and a unique
$Z_Y \in \widetilde{C}_{\mathfrak{h}}(X+Y)\cap C_{\mathfrak{h}}(Y)$
such that
$[X,Y]=[Z_Y,X]+[Z_X,Y]$.
\end{enumerate}
\end{theorem}

\begin{proof}
The implication $(2)\Longrightarrow (1)$ is obvious. Let us prove $(3)\Longrightarrow (2)$. Clear that for  $\lambda = \mu$ we have
a geodesic orbit metric. Now, fix some positive $\lambda \neq \mu$.
According to Lemma~\ref{GO-criterion1}, it suffices to find $Z \in \mathfrak{h}$ such that
$[X,Y]=\frac{\lambda}{\lambda-\mu}[Z,X]+\frac{\mu}{\lambda-\mu}[Z,Y]$. In fact,  for $\lambda X$ and $\mu Y$, there is a unique
$Z_X \in \widetilde{C}_{\mathfrak{h}}(X+Y)\cap C_{\mathfrak{h}}(X)$ and a unique
$Z_Y \in \widetilde{C}_{\mathfrak{h}}(X+Y)\cap C_{\mathfrak{h}}(Y)$
such that
$[\lambda X,\mu Y]=[Z_Y,\lambda X]+[Z_X,\mu Y]=[Z_X+Z_Y,\lambda X]+[Z_X+Z_Y, \mu Y]$ by~(3). Now, it suffices to take
$Z:=\frac{\lambda-\mu}{\lambda \cdot \mu}(Z_X+Z_Y)$ and apply Lemma \ref{GO-criterion1}.
\smallskip

Finally, we prove the implication $(1)\Longrightarrow (3)$. We will use the idea from the proof of Theorem 3 in \cite{Nik2016}.
By Lemma \ref{GO-criterion2}, for any $X \in \mathfrak{m}_1$ and any $Y \in \mathfrak{m}_2$ there is a unique
$Z \in \widetilde{C}_{\mathfrak{h}}(X+Y)$ such that
$[X,Y]=\frac{\lambda}{\lambda-\mu}[Z,X]+\frac{\mu}{\lambda-\mu}[Z,Y]$. By the same reason, for $X \in \mathfrak{m}_1$ and $-Y \in \mathfrak{m}_2$
there is a unique
$Z' \in \widetilde{C}_{\mathfrak{h}}(X+Y)$ such that
$[X,-Y]=\frac{\lambda}{\lambda-\mu}[Z',X]+\frac{\mu}{\lambda-\mu}[Z',-Y]$. From these two equalities we easily get
$\lambda[Z+Z',X]+\mu[Z-Z',Y]=0$, i.~e. $Z+Z' \in {C}_{\mathfrak{h}}(X)$ and $Z-Z' \in {C}_{\mathfrak{h}}(Y)$. On the other hand,
$Z+Z', Z-Z' \in \widetilde{C}_{\mathfrak{h}}(X+Y)$, since $Z$ and $Z'$ have this property.
Then $Z_X:=\frac{\mu}{2(\lambda-\mu)}(Z+Z')\in {C}_{\mathfrak{h}}(X) \cap \widetilde{C}_{\mathfrak{h}}(X+Y)$ and
$Z_Y:=\frac{\lambda}{2(\lambda-\mu)}(Z-Z')\in {C}_{\mathfrak{h}}(Y) \cap \widetilde{C}_{\mathfrak{h}}(X+Y)$. It is easy to check that
$Z=\frac{\lambda-\mu}{\mu}Z_X+\frac{\lambda-\mu}{\lambda}Z_Y$ and $[X,Y]=[Z_Y,X]+[Z_X,Y]$. The following is to prove the uniqueness of $Z_X$ and $Z_Y$.
If $[X,Y]=[Z'_Y,X]+[Z'_X,Y]$ for some
$Z'_X\in {C}_{\mathfrak{h}}(X) \cap \widetilde{C}_{\mathfrak{h}}(X+Y)$ and
$Z'_Y\in {C}_{\mathfrak{h}}(Y) \cap \widetilde{C}_{\mathfrak{h}}(X+Y)$, then $[Z_Y-Z'_Y,X]+[Z_X-Z'_X,Y]=0$, hence $Z_Y-Z'_Y\in {C}_{\mathfrak{h}}(X)$ and
$Z_X-Z'_X \in {C}_{\mathfrak{h}}(Y)$. On the other hand, $Z_X-Z'_X\in {C}_{\mathfrak{h}}(X) \cap \widetilde{C}_{\mathfrak{h}}(X+Y)$ and
$Z_Y-Z'_Y\in {C}_{\mathfrak{h}}(Y) \cap \widetilde{C}_{\mathfrak{h}}(X+Y)$, hence,
$Z_X-Z'_X,Z_Y-Z'_Y\in {C}_{\mathfrak{h}}(X+Y)\cap\widetilde{C}_{\mathfrak{h}}(X+Y)=\{0\}$, i.~e. the theorem holds.
\end{proof}

\begin{corollary}\label{cor2}
Suppose that the metric with $A_{\lambda,\mu}$, $\lambda \neq \mu$ is geodesic orbit. If $C_{\mathfrak{h}}(X)\subset C_{\mathfrak{h}}(Y)$ for some $X \in \mathfrak{m}_i$ and $Y \in \mathfrak{m}_j$ ($i\neq j$), then $[X,Y] \in \mathfrak{m}_i$. In particular, if $C_{\mathfrak{h}}(X)=0$, then
$[X,\mathfrak{m}_j] \subset \mathfrak{m}_i$.
\end{corollary}

\begin{proof} Without loss of generality we assume that $i=1$ and $j=2$.
Note that $C_{\mathfrak{h}}(X)\subset C_{\mathfrak{h}}(Y)$ is equivalent to $C_{\mathfrak{h}}(X)= C_{\mathfrak{h}}(X+Y)$.
By (3) in Theorem \ref{two modules}, we have that $Z_X=0$ and $[X,Y] \in \mathfrak{m}_1$.
\end{proof}
\smallskip

\begin{corollary}\label{cor3}
Suppose that the metric with $A_{\lambda,\mu}$, $\lambda \neq \mu$, is geodesic orbit. If there exists $X \in \mathfrak{m}_1$ and $Y \in \mathfrak{m}_2$ such that $[X,Y] \not\in  \mathfrak{m}_1$ and $[X,Y] \not\in  \mathfrak{m}_2$, then
$C_{\mathfrak{h}}(X)\neq 0$, $C_{\mathfrak{h}}(Y)\neq 0$, and $\dim \Bigl(\widetilde{C}_{\mathfrak{h}}(X+Y)\Bigr) \geq 2$.
\end{corollary}

\begin{proof} It suffices to note that
$Z_X$ and $Z_Y$ (see (3)) in Theorem \ref{two modules}) are non-trivial and linear independent.
\end{proof}

\begin{remark}\label{remdifmetgo}
According to Theorem \ref{two modules}, we call a compact homogeneous space $G/H$ with simple group $G$, connected $H$,  and two
irreducible components in the isotropy representation, {\it geodesic orbit} if it is geodesic orbit
with respect to any (hence, all) $G$-invariant Riemannian metric, that is not normal with respect to $G$ (i.~e. $\lambda \neq \mu$).
\end{remark}

Theorem \ref{two modules} allows us to apply special tools to study geodesic orbit Riemannian manifolds (compare with section 5 in \cite{Nik2016}).
These tools related to the problem of classifying principal orbit types for linear actions of
compact Lie groups. If a compact linear Lie group $K$ acts on some finite-dimensional vector space $V$
(in other terms, we have a representation of a Lie group on the space~$V$),
then almost all points of $V$ are situated on the orbits of $K$, that are pairwise isomorphic as $K$-manifolds.
Such orbits are called orbits in general position. The isotropy groups  of all points on such orbits are conjugate
in $K$, and the class of these isotropy groups is called a {\it principal isotropy group} for the linear group $K$
and the corresponding Lie algebra is called a {\it principal isotropy algebra} or a {\it stationary subalgebra of points in general position}.
Roughly speaking, the principal isotropy algebra is trivial for general linear Lie groups $K$, but it is not the case for some special linear groups.
\smallskip

The classification of linear actions of simple compact connected Lie groups with
non-trivial connected principal isotropy subgroups, has already
been carried out in \cite{Ela1}, \cite{Hsiang}, and  \cite{Kram}.
See details in \S 3 of Chapter 5 in \cite{Hsiangbook}.
In \S 4 and \S 5 of Chapter 5 in \cite{Hsiangbook}, one can find
a description of more general compact connected Lie groups with
non-trivial connected principal isotropy subgroups (see e.~g. Theorem (V.7') in \cite{Hsiangbook}).
\smallskip

Let $\chi: H \to O(\mathfrak{m})$ be the isotropy representation of a simply connected compact homogeneous space $G/H$ with simple $G$.
We suppose that $\chi =\chi_1\oplus \chi_2$, where $\chi_1$ and $\chi_2$ are irreducible subrepresentations on $\mathfrak{m}_i$ with
$\langle \mathfrak{m}_1, \mathfrak{m}_2 \rangle=0$.
If $G/H$ with the metric operator $A_{\lambda,\mu}$, $\lambda\neq \mu$, is geodesic orbit, then we have {\it non-trivial information on
the orbit types} of representations $\chi_1$ and $\chi_2$.
\smallskip

\begin{lemma}\label{lemmasimple}
In the above assumptions, we have $[\mathfrak{m}_1,\mathfrak{m}_2]\neq 0$.
\end{lemma}

\begin{proof}
Otherwise, $[\mathfrak{m}_1,\mathfrak{m}_2]=0$. Together with $\langle \mathfrak{m}_1,\mathfrak{m}_2 \rangle =0$, we have that $\mathfrak{m}_1+[\mathfrak{m}_1,\mathfrak{m}_1]$ and
$\mathfrak{m}_2+[\mathfrak{m}_2,\mathfrak{m}_2]$ are ideals (see details in Lemma 7 of \cite{Nik2015})
in the simple Lie algebra $\mathfrak{g}$, that is impossible.
Hence, $[\mathfrak{m}_1,\mathfrak{m}_2]\neq 0$.
\end{proof}

\begin{prop}\label{piar1}
If $[\mathfrak{m}_1,\mathfrak{m}_2]\not\subset \mathfrak{m}_1$ and  $[\mathfrak{m}_1,\mathfrak{m}_2]\not\subset \mathfrak{m}_2$, then
the principal isotropy group of the representations $\chi_1$ and $\chi_2$
has dimensions $\geq 1$, in particular, the principal isotropy algebras of these representations are not trivial.
\end{prop}

\begin{proof}
Let $\mathfrak{m}_{2,1}:=\{X\in \mathfrak{m}_2\,|\, [\mathfrak{m}_1, X]\subset \mathfrak{m}_1\}$ and
$\mathfrak{m}_{2,2}:=\{X\in \mathfrak{m}_2\,|\, [\mathfrak{m}_1, X]\subset \mathfrak{m}_2\}$. Clearly $\mathfrak{m}_{2,1}$ and $\mathfrak{m}_{2,2}$
are subspaces of $\mathfrak{m}_2$ of codimension
$\geq 1$. Fix $0\not=X\in \mathfrak{m}_2 \setminus (\mathfrak{m}_{2,1} \cup \mathfrak{m}_{2,2})$
and define $\mathfrak{m}_{1,1}(X):=\{Y\in \mathfrak{m}_1\,|\, [Y, X]\subset \mathfrak{m}_1\}$ and
$\mathfrak{m}_{1,2}(X):=\{Y\in \mathfrak{m}_1\,|\, [Y, X]\subset \mathfrak{m}_2\}$.
Similarly, we know that  $\mathfrak{m}_{1,1}$ and $\mathfrak{m}_{1,2}$
are subspaces of $\mathfrak{m}_1$ of codimension $\geq 1$.
Then for every $0\not=Y\in \mathfrak{m}_1 \setminus (\mathfrak{m}_{1,1} \cup \mathfrak{m}_{1,2})$, $[X,Y]\not \in \mathfrak{m}_1$
and $[X,Y]\not \in \mathfrak{m}_2$. By Theorem \ref{two modules}, there are a unique
$Z_X \in \widetilde{C}_{\mathfrak{h}}(X+Y)\cap C_{\mathfrak{h}}(X)$ and a unique
$Z_Y \in \widetilde{C}_{\mathfrak{h}}(X+Y)\cap C_{\mathfrak{h}}(Y)$
such that
$[X,Y]=[Z_Y,X]+[Z_X,Y]$. Clearly $Z_X\neq 0$ and $Z_Y \neq 0$. It implies that
$\dim (C_{\mathfrak{h}}(X))\geq 1$ and $\dim (C_{\mathfrak{h}}(Y))\geq 1$.
That is, the proposition holds.
\end{proof}

\begin{prop}\label{piar2}
If $[\mathfrak{m}_1,\mathfrak{m}_2]\subset \mathfrak{m}_2$ and $[\mathfrak{m}_1,\mathfrak{m}_2]\neq 0$, then
the principal isotropy group of the representation $\chi_1$
has dimensions $\geq 1$, in particular, the principal isotropy algebra of this representation is not trivial.
\end{prop}

\begin{proof}
By Theorem \ref{two modules}, for any $X\in \mathfrak{m}_2$, $Y\in \mathfrak{m}_1$,
there is a unique $Z_X \in \widetilde{C}_{\mathfrak{h}}(X+Y)\cap C_{\mathfrak{h}}(X)$ and a unique
$Z_Y \in \widetilde{C}_{\mathfrak{h}}(X+Y)\cap C_{\mathfrak{h}}(Y)$
such that
$[X,Y]=[Z_Y,X]+[Z_X,Y]$. Clearly $Z_X\in C_{\mathfrak{h}}(Y)$ (since $[X,Y]\subset \mathfrak{m}_2$), hence $Z_X=0$, and $[X,Y]=[Z_Y,X]$.
\smallskip

Let $\mathfrak{m}_2':=\{X\in \mathfrak{m}_2\,|\, [\mathfrak{m}_1, X]= 0\}$. Clearly $\mathfrak{m}_2'$ is a subspace of $\mathfrak{m}_2$ of codimension
$\geq 1$. Fix $X\in \mathfrak{m}_2 \setminus \mathfrak{m}_2'$ and define $\mathfrak{m}_1'=\mathfrak{m}_1(X):=\{Y\in \mathfrak{m}_1\,|\, [Y, X]= 0\}$.
Similarly $\mathfrak{m}_1'$ is a subspace of $\mathfrak{m}_1$ of codimension $\geq 1$.
Now for all $Y\in\mathfrak{m}_1 \setminus \mathfrak{m}_1'$, $[X,Y]\neq 0$. Thus $0\not=Z_Y \in  C_{\mathfrak{h}}(Y)$, therefore, $\dim (C_{\mathfrak{h}}(Y))\geq 1$.
It implies that the principal isotropy group of the representation $\chi_1$
has dimensions $\geq 1$, in particular, the principal isotropy algebra of this representation is not trivial.
\end{proof}

\begin{remark}
Note that Propositions \ref{piar1} and \ref{piar2}  follows also from Corollary 5 in \cite{Nik2016n}.
\end{remark}

Since the assertion of Proposition \ref{piar2} is less restrictive than assertion of Proposition~\ref{piar1}, we consider an additional construction.
Under the assumptions of Proposition \ref{piar2}, $\mathfrak{k}:=\mathfrak{h}\oplus \mathfrak{m}_1$ is a Lie subalgebra in $\mathfrak{g}$.
Denote by $K$ the corresponding connected subgroup in $G$. It could be defined as the unit component of the normalizer of $\mathfrak{k}$ in $G$,
hence, it is a closed subgroup.
Let us consider the isotropy representation $\eta: K \to O(\mathfrak{m}_2)$ for the homogeneous space $G/K$.

\begin{prop}\label{piar3}
If $[\mathfrak{m}_1,\mathfrak{m}_2]\subset \mathfrak{m}_2$ and $[\mathfrak{m}_1,\mathfrak{m}_2]\neq 0$, then
the principal isotropy group of the representation $\eta$
has dimension $\geq \dim (\mathfrak{m}_1)$, in particular, the principal isotropy algebra of this representation is not trivial.
\end{prop}

\begin{proof} For any $X \in \mathfrak{m}_2$, $Y \in \mathfrak{m}_1$, by the proof of Proposition \ref{piar2}, there is a unique
$Z_Y \in \widetilde{C}_{\mathfrak{h}}(X+Y)\cap C_{\mathfrak{h}}(Y)$ such that $[X,Y]=[Z_Y,X]$.
Therefore, $[Y+Z_Y, X]=0$ and $Y+Z_Y \in C_{\mathfrak{k}} (X)$, where $C_{\mathfrak{k}} (X)$ is the centralizer of $X$ in
the Lie algebra $\mathfrak{k}=\mathfrak{h}\oplus \mathfrak{m}_1$. It follows that $\dim C_{\mathfrak{k}} (X) \geq \dim(\mathfrak{m}_1)$.
That is, the principal isotropy group of the representation $\eta$ has dimension $\geq \dim (\mathfrak{m}_1)$.
\end{proof}

\begin{theorem}\label{theolast}
Assume that $G/H$ is a compact simply connected homogeneous space with $G$ simple and there are two components in the isotropy representation.
If $G/H$ is geodesic orbit, there exists $K\subset G$ such that $H\subset K\subset G$, $\dim H <\dim K <\dim G$, and $G/K$ is symmetric.
\end{theorem}
\begin{proof}
Firstly, we will prove that $H$ is not a maximal connected subgroup in $G$ under the assumption.
In Table 1, we list the cases given in \cite{DickKerr} such that
$H$ is maximal in $G$. These homogeneous spaces satisfy the conditions of Proposition~\ref{piar1}.
By principal isotropy algebras corresponding to the tables in Pages 197 and 202 of the paper \cite{Hsiang}, it is easy to see that
at least one of the principal isotropy algebras of the representations $\chi_1$ and $\chi_2$ is trivial. By Proposition~\ref{piar1}, it is impossible.
\smallskip

{\small
\renewcommand{\arraystretch}{1.5}
\begin{table}[t]
{\bf Table 1. }
\begin{center}
\begin{tabular}{|c|c|c|}
\hline
N & $H\subset G$ & $\chi=\chi_1+\chi_2 $   \\
\hline\hline
V.1 & $SO(m)\times SO(n)\subset SO(mn)$   & $\pi_2\otimes \pi_1^2\oplus \pi_1^2\otimes\pi_2$   \\
 & $(n,m\geq 3,n,m\not=4)$  &   \\ \hline
V.2 & $Sp(m)\times Sp(n)\subset SO(4mn)$   & $\pi_1^2\otimes\pi_2\oplus\pi_2\otimes \pi_1^2$  \\
 & $(n,m\geq 2)$  &   \\ \hline
V.3 & $Sp(6)\subset SO(429)$   & $\pi_5^2\oplus\pi_3^2$   \\ \hline
V.4 & $SU(12)\subset SO(924)$   & $\pi_3\pi_9\oplus\pi_5\pi_7$  \\ \hline
V.5 & $Spin(24)\subset SO(2048)$   & $\pi_6\oplus\pi_{10}$  \\ \hline
V.6 & $SU(n)\subset SU\left(\frac{n(n+1)(n+2)}{6}\right)$   & $\pi_1^2\pi^2_{n-1}\oplus \pi_1^3\pi^3_{n-1}$  \\
 & $n\geq 2$  &   \\ \hline
V.7 & $SU(n)\subset SU\left(\frac{n(n-1)(n-2)}{6}\right)$   & $\pi_2\pi_{n-2}\oplus \pi_3\pi_{n-3}$   \\
 & $n\geq 6$  &   \\ \hline
V.8 & $Spin(14)\subset SU(64)$   & $\pi_4\oplus\pi_6\pi_7$ \\ \hline
V.9 & $SO(m)\times Sp(n)\subset Sp(mn)$   & $\pi_1^2\otimes\pi_1^2\oplus\pi_2\otimes \pi_2$  \\
 & $(m\geq 3,m\not=4,n\geq 2)$  &   \\ \hline
V.10 & $Sp(1)\subset Sp(3)$   & $\pi^6\oplus\pi^{10}$  \\ \hline
V.11 & $Spin(13)\subset Sp(32)$   & $\pi_3\oplus\pi_{6}^2$  \\ \hline
V.12 & $Spin(20)\subset Sp(256)$   & $\pi_6\oplus\pi_{10}^2$   \\ \hline
V.13 & $SU(10)\subset Sp(126)$   & $\pi_5^2\oplus\pi_{3}\pi_7$  \\ \hline
V.14 & $Sp(4)\subset Sp(24)$   & $\pi_3^2\oplus\pi_{1}\pi_3$   \\ \hline
V.15 & $Sp(5)\subset Sp(66)$   & $\pi_3^2\oplus\pi_{5}^2$   \\ \hline
V.16 & $SU(2)\times G_2\subset E_7$   & $\pi^2\otimes\pi_2^2\oplus\pi^4\otimes\pi_{2}$  \\ \hline
V.17 & $SU(5)SU(5)\subset E_8$   & $[\pi_1\otimes\pi_3]_{\mathbb R}\oplus[\pi_2\otimes\pi_1]_{\mathbb R}$  \\ \hline
V.18 & $Sp(2)\subset E_8$   & $\pi_1^2\pi_2^3\oplus\pi_1^6$  \\ \hline
\end{tabular}
\end{center}
\end{table}
}

That is, $H$ is not maximal connected in $G$. Then there exists connected subgroup $K\subset G$ such that $H\subset K\subset G$,
$\dim H <\dim K <\dim G$.
In Table 2, we will list the cases in \cite{DickKerr} satisfying $\mathfrak{k}:= \mathfrak{h}\oplus \mathfrak{m}_1$ is a subalgebra
of $\mathfrak{g}$ and $(\mathfrak{g},\mathfrak{k})$ is not a symmetric pair but irreducible.

{\small
\renewcommand{\arraystretch}{1.4}
\begin{table}[t]
{\bf Table 2.}
\begin{center}
\begin{tabular}{|c|c|c|}
\hline
N & $H\subset K\subset G$ &  $\dim {\mathfrak m_1}$  \\
\hline\hline
I.20 & $Sp(1)\times Sp(3)\subset  Sp(1)\times Sp(7)\subset SO(28)$ &    84 \\ \hline
I.21 & $Sp(1)\times SU(6)\subset  Sp(1)\times Sp(10)\subset SO(40)$ &   175  \\ \hline
I.22 & $Sp(1)\times Spin(12)\subset  Sp(1)\times Sp(16)\subset SO(64)$ &    462  \\ \hline
I.23 & $Sp(1)\times E_7\subset  Sp(1)\times Sp(28)\subset SO(112)$ &   1463  \\ \hline
I.29 & $SO(3)\subset  G_2\subset SO(7)$ &   11 \\ \hline
III.9 & $G_2\times Sp(1)\subset  SO(7)\times Sp(1)\subset  Sp(7) $    & 7  \\ \hline
III.10 & $Spin(7)\times Sp(1)\subset  SO(8)\times Sp(1)\subset  Sp(8) $   & 7 \\ \hline
III.11 & $Spin(11)\subset  Spin(12)\subset  Sp(16) $ &  11 \\ \hline
IV.3 & $SO(3)SU(3)\subset  (SU(3))^2\subset  F_4$ & 5 \\ \hline
IV.6 & $SU(2)\times SO(3)\subset  SU(2)\times G_2\subset  F_4$ & 11 \\ \hline
IV.13 & $SO(3)(SU(3))^2\subset  (SU(3))^3\subset  E_6$ &    5 \\ \hline
IV.18 & $SU(3)SO(3)\subset  SU(3)G_2\subset  E_6$ &   11 \\ \hline
IV.30 & $SO(3)SU(6)\subset  SU(3)SU(6)\subset  E_7$ &    5 \\ \hline
IV.31 & $SU(3)SO(6)\subset  SU(3)SU(6)\subset  E_7$ &    20 \\ \hline
IV.32 & $SU(3)SU(3)\subset  SU(3)SU(6)\subset  E_7$ &   27  \\ \hline
IV.33 & $SO(3)Sp(3)\subset  G_2Sp(3)\subset  E_7$ &   11 \\ \hline
IV.41 & $SO(3)E_6\subset  SU(3)E_6\subset  E_8$ &     5\\ \hline
IV.42 & $SU(3)\times Sp(4)\subset  SU(3)\times E_6\subset  E_8$ &    42 \\ \hline
IV.43 & $SU(3)\times G_2\subset  SU(3)\times E_6\subset  E_8$ &   64
 \\  \hline
IV.44 & $SU(3)\times SU(3)\subset  SU(3)\times E_6\subset  E_8$   & 70 \\ \hline
IV.48 & $SO(3)\times F_4\subset  G_2\times F_4\subset  E_8$ & 11 \\ \hline
\end{tabular}
\end{center}
\end{table}
}
\smallskip

Directly by Propositions~\ref{piar2} and~\ref{piar3}, $\dim \mathfrak m_1$ in Table 2 and the principal isotropy group of the representation in Pages 197 and 202
of the paper \cite{Hsiang}, it is easy to know that all the cases in Table 2 are not geodesic orbit.
\smallskip

In fact, by Proposition~\ref{piar2} and the principal isotropy group of the representation of $Spin(11)$ in Pages 197 of the paper \cite{Hsiang}, we know that
$III.11$ is not geodesic orbit.
\smallskip

If $K=Sp(n)\times Sp(1)$, the maximal dimension of possible principal isotropy group of the representation of $K$ is $(n-1)(2n-1)+3=\dim Sp(n-1)+3+\dim Sp(1)$.
Comparing to $\dim \mathfrak m_1$ in Table 2, by proposition~\ref{piar3}, $I.20$ and $I.21$ are not geodesic orbit. For this compact semisimple Lie group,
if the maximal dimension of possible principal isotropy group of the representation of $K$ is non-trivial. Then the representation of $Sp(n)$ must be $\mu_n$
with the dimension $2n$. Thus the representation of $K$ on $\mathfrak m_2$ must the multiple of $2n$. For $I.22$ and $I.23$, the dimensions are 1485 and 4617.
That is, $I.22$ and $I.23$ are not geodesic orbit.
\smallskip

If $K=G_2$, the maximal dimension of possible principal isotropy group of the representation of $K$ is $8=\dim SU(3)$. But $\dim {\mathfrak m_1}=11$ by Table 2.
By proposition~\ref{piar3}, $I.29$ is not geodesic orbit.
\smallskip

If $K=SO(n)\times SO(3)$, the maximal dimension of possible principal isotropy group of the representation of $K$ is $3=\dim SU(2)$. Comparing to
$\dim \mathfrak m_1$ in Table 2, by proposition~\ref{piar3}, $III.9$ and $III.10$ are not geodesic orbit.
\smallskip

If $K=(SU(3))^3$, the maximal possible principal isotropy group of the representation of $K$ is $SU(2)$, which is 3-dimensional.
Comparing to $\dim \mathfrak m_1$ in Table 2, by proposition~\ref{piar3}, $IV.13$ is not geodesic orbit.
\smallskip

If $K=SU(n)\times G_2$, the maximal possible principal isotropy group of the representation of $K$ is $SU(n-1)$. Comparing to $\dim \mathfrak m_1$ in Table 2,
by proposition~\ref{piar3}, $IV.6$ and $IV.18$ are not geodesic orbit.
\smallskip

If $K=SU(3)\times SU(6)$, the maximal possible principal isotropy group of the representation of $K$ is $SU(3)\times T^2$, which is 10-dimensional.
Comparing to $\dim \mathfrak m_1$ in Table 2, by proposition~\ref{piar3}, $IV.30$, $IV.31$ and $IV.32$ are not geodesic orbit.
\smallskip

If $K=G_2\times Sp(3)$, the maximal possible principal isotropy group of the representation of $K$ is $Sp(2)$, which is 10-dimensional.
Comparing to $\dim \mathfrak m_1$ in Table 2, by proposition~\ref{piar3}, $IV.33$ is not geodesic orbit.
\smallskip

If $K=SU(3)\times E_6$, the maximal possible principal isotropy group of the representation of $K$ is $SU(2)$, which is 3-dimensional.
Comparing to $\dim \mathfrak m_1$ in Table 2, by proposition~\ref{piar3}, $IV.41$, $IV.42$, $IV.43$ and $IV.44$ are not geodesic orbit.
\smallskip

Since $G_2\times F_4$ doesn't appear in the table given in paper \cite{Hsiang}, we know that the principal isotropy group must be trivial.
It follows that $IV.48$ is not geodesic orbit.

In summary, we have the theorem.
\end{proof}

\smallskip

{\small
\renewcommand{\arraystretch}{1.3}
\begin{table}[t]
{\bf Table 3. The classification in \cite{Ta}.}
\begin{center}
\begin{tabular}{|c|c|c|c|}
\hline
$\mathfrak g$ & $\mathfrak k$ & $\mathfrak h$ &  \\
\hline\hline
$so(2n+1)$ & $so(2n)$ & $u(n)$ & $n\geq 2$   \\ \hline
$so(4n+1)$ & $so(4n)$ & $su(2n)$ & $n\geq 1$   \\ \hline
$so(8)$ & $so(7)$ & $g_2$ &    \\ \hline
$so(9)$ & $so(8)$ & $so(7)$ &    \\ \hline
$su(n+1)$ & $u(n)$ & $su(n)$ & $n\geq 2$   \\ \hline
$su(2n+1)$ & $u(2n)$ & $u(1)\oplus sp(n)$ & $n\geq 2$   \\ \hline
$su(2n+1)$ & $u(2n)$ & $sp(n)$ & $n\geq 2$   \\ \hline
$sp(n+1)$ & $sp(1)\oplus sp(n)$ & $u(1)\oplus sp(n)$ & $n\geq 1$   \\ \hline
$sp(n+1)$ & $sp(1)\oplus sp(n)$ & $sp(n)$ & $n\geq 1$   \\ \hline
$su(2r+n)$ & $su(r)\oplus su(r+n)\oplus \mathbb{R}$ & $su(r)\oplus su(r+n)$ & $r\geq 2,n\geq 1$   \\ \hline
$so(4r+2)$ & $u(2r+1)$ & $su(2r+1)$ & $r\geq 2$   \\ \hline
$e_6$ & $so(10)\oplus so(2)$ & $so(10)$ &    \\ \hline
$so(9)$ & $so(7)\oplus so(2)$ & $g_2\oplus so(2)$ &    \\ \hline
$so(10)$ & $so(8)\oplus so(2)$ & $spin(7)\oplus so(2)$ &    \\ \hline
$so(11)$ & $so(8)\oplus so(3)$ & $spin(7)\oplus so(3)$ &    \\ \hline
\end{tabular}
\end{center}
\end{table}
}
\smallskip

By Theorem~\ref{theolast}, we only need to discuss the case that there exists $K\subset G$ such that $H\subset K\subset G$ and $G/K$ is symmetric. In general, let $\mathfrak h\subset\mathfrak k$ be Lie algebras of $H\subset K$ respectively. Then we have the following
decomposition: $$\mathfrak g=\mathfrak k\oplus \mathfrak m_2=\mathfrak h\oplus \mathfrak m_1\oplus \mathfrak m_2.$$
Consider the following $G$-invariant inner product on $G/H$:
$$
g_{a,b}:=a\,\langle\cdot,\cdot\rangle|_{\mathfrak m_1}+b\,\langle\cdot,\cdot\rangle|_{\mathfrak m_2},
$$
which is as in (\ref{twopar1}).
Here if $g_{a,b}$ is geodesic orbit for $a\not=b$, then $g_{a,b}$ is geodesic orbit for any $a,b>0$ \cite{Ta} (this is generalized in Theorem \ref{two modules}).
Also the author gives in \cite{Ta}
the classification of $G/H$ satisfying that $G/H$ is a compact effective irreducible symmetric pair and the
metric $g_{a,b}$ is geodesic orbit for any $a,b>0$. The classification in \cite{Ta} doesn't need the condition that there are only two components
in the isotropy representation. Here we list the classification in \cite{Ta} in Table~3.
\smallskip

{\bf The proof of Theorem~\ref{maimres}.} According to Theorem \ref{theolast}, we only need to discuss the case that there exists $K\subset G$ such that
$H\subset K\subset G$ and $G/K$ is symmetric.
In order to complete Theorem~\ref{maimres}, we need to choose the homogeneous spaces from Table 3 such that there are only two components in the isotropy
representation. Comparing with the classification given in \cite{DickKerr} and taking into account Theorem \ref{two modules} and Proposition \ref{prop.zil},
we have Theorem~\ref{maimres}.
\medskip

{\bf Acknowledgements.}
The authors are indebted to Prof. Megan Kerr for helpful discussions concerning this paper. The work is partially supported by Grant 1452/GF4 of Ministry of Edu\-ca\-tion and Sciences of the Republic of Kazakhstan for 2015\,-2017 and NSF of China (No.~11571182).

\bigskip

\bibliographystyle{amsunsrt}

\begin{thebibliography}{[99]}


\bibitem{AA}
D.V.~Alekseevsky, A.~Arvanitoyeorgos,
{\sl Riemannian flag manifolds with homogeneous geodesics,} Trans. Amer. Math. Soc., 359 (2007), 3769--3789.

\bibitem{AN}
D.V. Alekseevsky, Yu.G. Nikonorov,  {\sl Compact Riemannian manifolds with homogeneous
geodesics}, SIGMA Symmetry Integrability Geom. Methods Appl., 5 (2009), 093, 16 pages.


\bibitem{AV}
D.N.~Akhiezer, \'{E}.B.~Vinberg,
{\sl Weakly symmetric spaces and spherical varieties,} Transform. Groups, 4 (1999), 3--24.


%\bibitem{Baba2008}
%K. Baba,
%{\sl Local orbit types of s-representations for exceptional semisimple symmetric spaces,}
%SUT Journal of Mathematics
%Vol. 44(2), (2008), 307--328.




\bibitem{BerNik}
V.N.~Berestovskii, Yu.G.~Nikonorov,
{\sl On $\delta$-homogeneous Riemannian manifolds,}
Differential Geom. Appl., 26(5) (2008), 514--535.


\bibitem{BerNik3}
V.N.~Berestovskii, Yu.G.~Nikonorov,
{\sl On $\delta$-homogeneous Riemannian manifolds, II,}
Sib. Math. J., 50(2) (2009), 214--222.



\bibitem{BerNikClif}
V.N. Berestovskii, Yu.G. Nikonorov, {\sl Clifford-Wolf homogeneous Riemannian manifolds},
J. Differential Geom., 82(3) (2009), 467--500.


\bibitem{BerNik2012}
V.N.~Berestovskii, Yu.G.~Nikonorov,
{\sl Generalized normal homogeneous Riemannian metrics on spheres and projective spaces,}
Ann. Global Anal. Geom., 45(3) (2014), 167--196.

\bibitem{BKV}
J.~Berndt, O.~Kowalski, L.~Vanhecke,
{\sl Geodesics in weakly symmetric spaces,}
Ann. Global Anal. Geom. 15 (1997), 153--156.


\bibitem{Bes}
A.L.~Besse,
{\sl Einstein Manifolds}, Springer-Verlag, Berlin,
Heidelberg, New York, London, Paris, Tokyo, 1987.

%\bibitem{Bourb}
%N.~Bourbaki
%{\sl Lie groups and Lie algebras}.
%Chapters 1--3. Translated from the French. Reprint of the 1989 English translation. Springer-Verlag, Berlin, 1998.


\bibitem{CNN2017}
Z. Chen, Yu.G. Nikonorov, Yu.V. Nikonorova,
{\sl  Invariant Einstein metrics on Ledger--Obata spaces,} Differential geometry and its applications, 50 (2017), 71--87.




\bibitem{DickKerr}
W.~Dickinson and M.~Kerr,
{\sl The geometry of compact homogeneous spaces with two isotropy summands},
Ann. Glob. Anal. Geom. 34 (2008), 329--350.


\bibitem{DuKoNi}
Z.~Du{\u s}ek, O.~Kowalski, S.~Nik{\u c}evi{\'c},
{\sl New examples of Riemannian g.o. manifolds in dimension 7},
Differential Geom. Appl., 21 (2004), 65--78.

\bibitem{Ela1}
A. G. \'{E}lashvili,
{\sl Canonical form and stationary subalgebras of points in general position for simple linear Lie groups},
Funktsional. Anal. i Prilozen. 6(1) (1972), 51-62 (Russian); English translation: Funct. Anal. Appl. 6(1), (1972), 44--53




\bibitem{Gor96}
C.~Gordon,
{\sl Homogeneous Riemannian manifolds whose geodesics are orbits,} 155--174.
In: Progress
in Nonlinear Differential Equations. V.~20. Topics in geometry: in
memory of Joseph D'Atri. Birkh{\"a}user, 1996.


%\bibitem{Hel}
%S.~Helgason,
%{\sl Differential geometry and symmetric spaces,}
%Academic Press Inc., New-York, 1962.

\bibitem{Hsiang}
W.~C.~Hsiang, W.~Y.~Hsiang,
{\sl Differentiable actions of compact connected classical group}
%I, Amer. J. Math., 89 (1967), 705--786;
II, Ann. of Math., 92 (1970), 189--223,

\bibitem{Hsiangbook}
W.~Y.~Hsiang,
{\sl Cohomology theory of topological transformation groups},
Springer Verlag, New York, Heidelberg, Berlin, 1975, X+164 pp.


%\bibitem{Jac}
%N.~Jacobson,
%{\sl Lie algebras,}
%Interscience Publishers, New York-London, 1962.


\bibitem{Kap}
A.~Kaplan,
{\sl On the geometry of groups of  Heisenberg type},
Bull. London Math. Soc., 15 (1983), 35--42.

\bibitem{Kerr98}
M. Kerr,
{\sl New examples of homogeneous Einstein metrics,}
Mich. Math. J. 45(1) (1998), 115--134.


\bibitem{KN}
S.~Kobayashi, K.~Nomizu,
{\sl Foundations of differential
geometry,} Vol.~I -- A Wiley-Interscience Publication, New York,
1963; Vol.~II -- A Wiley-Interscience Publication, New York, 1969.

%\bibitem{Kondo2003}
%K. Kondo,
%{\sl Local orbit types of s-representations of symmetric R-spaces},
%Tokyo J. Math. 26(1), 2003, 67--81.

\bibitem{KV}
O.~Kowalski, L.~Vanhecke,
{\sl Riemannian manifolds with homogeneous geodesics,}
Boll. Un. Mat. Ital. B (7),  5(1)  (1991),  189--246.


\bibitem{Kram}
M.~Kr\"{a}mer,
{\sl Hauptisotropiegruppen bei endlich dimensionalen Darstellungen kompakter halbeinfacher Liegruppen}, Diplomarbeit,
Bonn, 1966.


\bibitem{LO1968}
A.J. Ledger, M. Obata,
{\sl Affine and Riemannian s-manifolds}, J. Differential Geom. 2 (1968), 451--459.



\bibitem{Nik2015}
Yu.G.~Nikonorov,
{\sl Killing vector fields of constant length on compact homogeneous Riemannian manifolds}, Ann. Glob. Anal. Geom., 48(4) (2015), 305--330.


\bibitem{Nik2016n}
Yu.G.~Nikonorov,
{\sl Classification of generalized Wallach spaces}, Geometriae Dedicata, 181(1) (2016), 193--212.

\bibitem{Nik2016}
Yu.G. Nikonorov,
{\sl On the structure of geodesic orbit Riemannian spaces},
Preprint, 2016, arXiv:1611.010501.

\bibitem{Ngu}
H.D.~Nguy\^{e}\~{n},
{\sl Compact weakly symmetric spaces and spherical pairs}, Proc. Amer. Math. Soc., 128(1) (2000), 3425--3433.


\bibitem{S}
A.~Selberg,
{\sl Harmonic Analysis and discontinuous groups in weakly symmetric Riemannian spaces,
with applications to Dirichlet series,} J. Indian Math. Soc. 20 (1956), 47--87.


%\bibitem{Taft}
%E.J.~Taft,
%{\sl Orthogonal conjugacies in associative and Lie algebras}, Trans. Amer. Math. Soc. 113 (1964), 18--29


\bibitem{Tam}
H.~Tamaru,
{\sl Riemannian geodesic orbit space metrics on fiber bundles,}
Algebras Groups Geom., 15 (1998) 55--67.


\bibitem{Ta}
H.~Tamaru,
{\sl Riemannian g.o. spaces fibered over irreducible symmetric spaces,}
Osaka J. Math.,  36 (1999), 835--851.



%\bibitem{Tamaru1999}
%H.~Tamaru,
%{\sl The local orbit types of symmetric spaces under the actions of the isotropy subgroups,}
%Differential Geom. Appl. 11(1) (1999),  29--38.

%\bibitem{Tamaru2003}
%H.~Tamaru,
%{\sl Two-step nilpotent Lie groups and homogeneous fiber bundles,}
%Ann. Global Anal. Geom. 24(1) (2003),  53--66.

%\bibitem{WZ3}
%M.~Wang, W.~Ziller,
%{\sl Existence and  non-existence of homogeneous Einstein metrics},
%Invent. Math., 84, (1986), 177--194.

%\bibitem{Wolf1969}
%J.A.~Wolf, {\sl A compatibility condition between invariant riemannian metrics and Levi--Whitehead decompositions on a coset space},
%Trans. Amer. Math. Soc. 139 (1969), 429--442.

\bibitem{W1}
J.A.~Wolf,
{\sl Harmonic Analysis on Commutative Spaces,} American Mathematical Society, 2007.


\bibitem{Yakimova}
O.S.~Yakimova.
{\sl Weakly symmetric Riemannian manifolds with a reductive isometry group},
(Russian)  Mat. Sb.  195(4)  (2004), 143--160;  English translation in Sb. Math.  195(3--4)  (2004),  599--614.



\bibitem{Z2}
W.~Ziller
{\sl The Jacobi equation on naturally reductive compact Riemannian homogeneous spaces}, Comm. Math. Helv., 52 (1977), 573--590.

\bibitem{Zil96}
W.~Ziller,
{\sl Weakly symmetric spaces,} 355--368. In: Progress
in Nonlinear Differential Equations. V.~20. Topics in geometry: in
memory of Joseph D'Atri. Birkh{\"a}user, 1996.

\end{thebibliography}

\vspace{10mm}

\end{document}